\documentclass[twoside,a4paper,11 pt]{amsart}

\usepackage{wrapfig}
\usepackage{amssymb,amsmath}
\usepackage{graphicx}
\usepackage[lite]{amsrefs}
\usepackage[all]{xy}
\usepackage{geometry}
\usepackage{tikz}
\usetikzlibrary{matrix,arrows}

\newcommand{\ringO}{\mathcal{O}}
\newcommand{\Hy}{\mathcal{H}}
\newcommand{\C}{{\mathbb{C}}}
\newcommand{\R}{{\mathbb{R}}}
\newcommand{\Z}{{\mathbb{Z}}}

\newcommand{\T}{{\bf{T}}}
\newcommand{\rationals}{{\mathbb{Q}}}

\newcommand*{\Homol}{\operatorname{H}}
\newcommand*{\SL}{\operatorname{SL}}

\theoremstyle{plain}
\newtheorem{theorem}{Theorem}
\newtheorem{lemma}[theorem]{Lemma}

\theoremstyle{definition}
\newtheorem{proposition}[theorem]{Proposition}
\newtheorem{corollary}[theorem]{Corollary}

\theoremstyle{remark}

\newtheorem{remarks}[theorem]{\bfseries Remarks}

\begin{document}

\title{On a question of Serre}
\author{Alexander D. Rahm}
\address{National University of Ireland at Galway, Department of Mathematics}
\urladdr{http://www.maths.nuigalway.ie/\~{}rahm/}
\email{Alexander.Rahm@nuigalway.ie}
\thanks{Funded by the Irish Research Council. \hfill \textit{E-mail address:} Alexander.Rahm@nuigalway.ie}

\date{\today}

\begin{abstract}

 Consider an imaginary quadratic number field $\rationals(\sqrt{-m})$, with $m$ a square-free positive integer, 
and its ring of integers $\ringO$.
The \emph{\mbox{Bianchi} groups} are the groups $\mathrm{SL_2}(\ringO)$.
Further consider the Borel--Serre compactification \cite{Serre} of the quotient of hyperbolic 
$3$--space $\Hy$ by a finite index subgroup $\Gamma$ in a Bianchi group,
and in particular the following question which Serre posed on page 514 of the quoted article.
Consider the map~$\alpha$
 induced on homology when attaching the boundary
 into the Borel--Serre compactification.

\emph{How can one determine the kernel of $\alpha$ (in degree 1) ?}

Serre used a global topological argument and obtained the rank of the kernel of $\alpha$.
In the quoted article, Serre did add the question what submodule precisely this kernel is.
Through a local topological study, we can decompose the kernel of~$\alpha$ into its parts associated to each cusp.
\end{abstract}

 \maketitle

\footnotesize
 \textbf{ R\'esum\'e.} Consid\'erons un corps quadratique imaginaire $\rationals(\sqrt{-m})$, o\`u $m$ est un entier positif ne contenant pas de carr\'e, 
et son anneau d'entiers $\ringO$.
Les \emph{groupes de \mbox{Bianchi}} sont les groupes $\mathrm{SL_2}(\ringO)$.
Puis, nous consid\'erons la compactification de Borel--Serre \cite{Serre} du quotient de l'espace hyperbolique $\Hy$ \`a trois dimensions 
par un sous-groupe $\Gamma$ d'indice fini dans un groupe de Bianchi,
et en particulier la question suivante que Serre posait sur la page 514 de l'article cit\'e.
Consid\'erons l'application~$\alpha$ induite en homologie quand le bord  est attach\'e dans la compactification de Borel--Serre.

\emph{Comment peut-on d\'eterminer le noyau de $\alpha$ (en degr\'e 1) ?} 

Serre se servait d'un argument topologique global et obtenait le rang du noyau de $\alpha$.
Dans l'article cit\'e, Serre rajoutait la question de quel sous-module pr\'ecis\'ement il s'agit pour ce noyau.
 A travers d'une \'etude topologique locale, nous pouvons d\'ecomposer le noyau de~$\alpha$
dans ses parties associ\'ees \`a chacune des pointes. 
\normalsize 

\medskip

We will exclusively consider the Bianchi groups \emph{with only units} $\{\pm 1\}$ in $\ringO$,
 by which we mean that we exclude the two Bianchi groups where $\ringO$ are the Gauss{}ian or Eisenstein integers.
Those two special cases can easily be treated separately. 
A wealth of information on the \mbox{Bianchi} groups can be found in the monographs \cite{ElstrodtGrunewaldMennicke}, \cite{Fine}, \cite{MaclachlanReid}.
For a subgroup $\Gamma$ of finite index in $\mathrm{SL_2}(\ringO)$,
 consider the Borel--Serre compactification $_\Gamma \backslash \widehat{\Hy}$ of the orbit space $_\Gamma \backslash \Hy$,
constructed in the appendix of \cite{Serre}.
As a fundamental domain in hyperbolic space, we make use of the polyhedron with missing vertices at the cusps,
described by Bianchi \cite{Bianchi}, and which we will call the \emph{Bianchi fundamental polyhedron}.
Our main result is the following, for which we consider the cellular structure on $_\Gamma \backslash \widehat{\Hy}$
induced by the Bianchi fundamental polyhedron.

\begin{theorem} \label{torus}
 The boundary $\partial(_\Gamma \backslash \widehat{\Hy})$ 
 is included as a sub-cellular
chain complex into the Borel--Serre compactification $_\Gamma \backslash \widehat{\Hy}$ in the following way.
\begin{enumerate}
 \item[(0)] All vertices of $_\Gamma \backslash \widehat{\Hy}$ are equivalent modulo the image of the $1$--cells 
(they define the same class in degree 0 homology).
 \item[(1)] For each orbit of cusps, exactly one of the two attached $1$--cells is the boundary of a $2$--chain.
 \item[(2)] The boundary of the Bianchi fundamental polyhedron is the union over the attached $2$--cells.
\end{enumerate}
\end{theorem}

\begin{proof} ${}$

\begin{enumerate}
 \item[(0)] is obvious since the orbit space $_\Gamma \backslash \widehat{\Hy}$ is path-wise connected.
 \item[(1)] will be proved cusp-wise in Lemma~\ref{torus-lemma}. 
 \item[(2)] follows from Poincar\'e's theorem on fundamental polyhedra \cite{Poincare},
  which tells us that the 
  $2$--dimensional facets of the Bianchi fundamental polyhedron inside hyperbolic space appear in pairs modulo the action of~$\Gamma$,
  with opposite signs. 
  The $2$--cells which we attach at the cusps are by construction unique modulo the action of~$\Gamma$.
\end{enumerate}
\end{proof}

Throughout this article, we use the upper-half space model $\Hy$ for hyperbolic $3$-space, as it is the one used by Bianchi.
As a set, $$ \Hy = \{ (z,\zeta) \in \C \times \R \medspace | \medspace \zeta > 0 \}. $$
We will call the \emph{real upper-half plane} the subset of $\Hy$ such that the coordinate $z$ has vanishing imaginary part,
and the \emph{imaginary upper-half plane} the subset of $\Hy$ such that the coordinate $z$ has vanishing real part.
Recall \cite{Serre} that the Borel--Serre compactification joins a $2$-torus $\T$
to $_\Gamma \backslash \Hy$ at every cusp.
We decompose $\T$ in the classical way into a $2$--cell, two $1$--cells and a vertex (see figure~\ref{torus-cells}).

\begin{lemma} \label{torus-lemma}
 The inclusion of~$\T$ into the Borel--Serre compactification of~$_\Gamma \backslash \Hy$ 
 makes exactly one of the $1$--cells of~$\T$ become the boundary of a $2$--chain.
\end{lemma}

\begin{proof}
 Consider the fundamental rectangle $\mathcal F$ for the action of the cusp stabiliser on the plane joined to $\Hy$ at our cusp.
 There is a sequence of rectangles in $\Hy$ obtained as translates of~$\mathcal F$ orthogonal to all the geodesic arcs emanating from the cusp.
 This way, the portion of the fundamental polyhedron which is nearest to the cusp, is trivially foliated by
 (locally homeomorphic to the Cartesian product of a geodesic arc with) translates of~$\mathcal F$
 (see figure \ref{diagram6} for the case $\Gamma = \SL_2({\mathcal{O}}_{\rationals(\sqrt{-6} \thinspace)})$,
 where the fundamental polyhedron admits one cusp at $\infty$ and one cusp at $\left. \frac{\sqrt{-6}}{2}\right)$.
 The boundaries of these translates are subject to the same identifications by $\Gamma$ as the boundary of~$\mathcal F$.
 Namely, denote by $t({\mathcal F})$ one of the translates of~$\mathcal F$; and denote by $\gamma_x$ and $\gamma_y$
 two generators, up to the $-1$ matrix, of the cusp stabiliser identifying the opposite edges of~$\mathcal F$
 (for a full Bianchi group $\SL_2(\Z[\omega])$, we can choose 
$\gamma_x =$\scriptsize$\begin{pmatrix} 1 & 1 \\ 0 & 1 \end{pmatrix}$ \normalsize and
$\gamma_y =$\scriptsize$\begin{pmatrix} 1 & \omega \\ 0 & 1 \end{pmatrix}$ \normalsize
for the cusp at~$\infty$, and the other cusp stabilisers are conjugate under $\SL_2(\C)$). 
Then,  $\gamma_x$ and $\gamma_y$ identify the opposite edges of $t({\mathcal F})$
 and make the quotient of the latter rectangle into a torus.
 So in the quotient space by the action of  $\Gamma$, the image of~$\T$ is wrapped into a sequence of layers of tori.
 And therefore in turn,
 the 3-dimensional interior of the Bianchi fundamental polyhedron is wrapped around the image of~$\T$
 along the entire surface of the latter.
 Hence, there is a neighbourhood of the image of~$\T$
 that is homeomorphic to Euclidean $3$--space with the interior of a solid torus removed.
 Now, considering the cell structure of the torus (see figure~\ref{torus-cells}),
 we see that precisely one of the loops generating the fundamental group of~$\T$
 (without loss of generality, the edge that has its endpoints identified by $\gamma_x$) 
 can be contracted in the interior of the image of the Bianchi fundamental polyhedron.
 The other loop (the one obtained by the identification by  $\gamma_y$)
 is linked with the removed solid torus and thus remains uncontractible in the Borel--Serre compactification
 of~$_\Gamma \backslash \Hy$.
\end{proof}

\begin{figure}
\includegraphics[height=60mm]{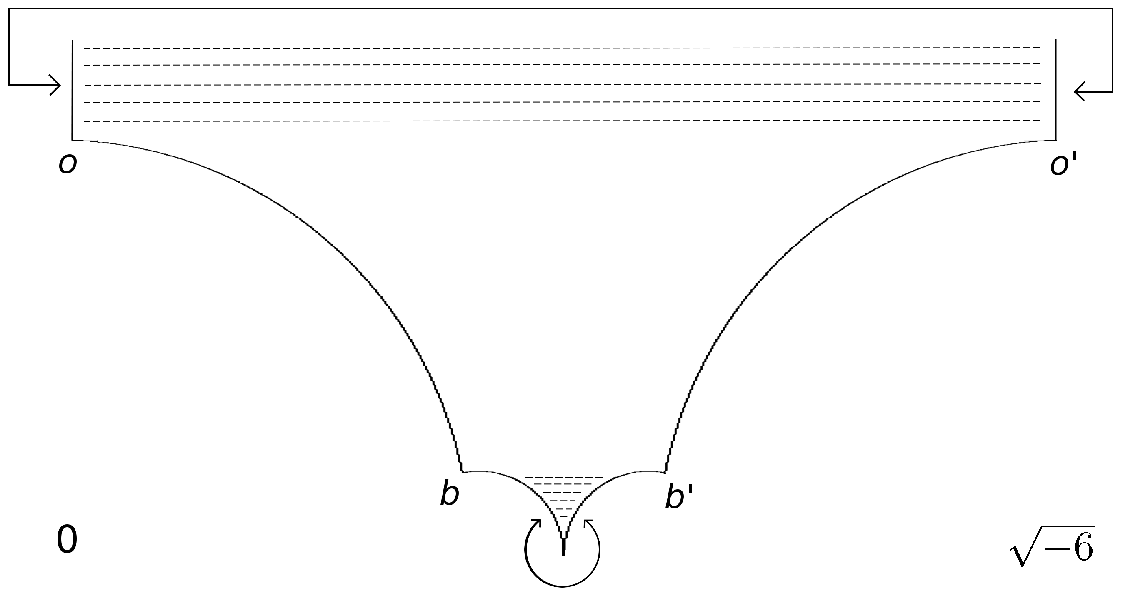}
\caption{Intersection of the fundamental polyhedron for SL$_2({\mathcal{O}}_{\rationals(\sqrt{-6} \thinspace)})$
 with the \mbox{imaginary} upper half-plane. The intersections with the translates of~$\mathcal F$ are indicated by the dashed lines. }
\label{diagram6}
\end{figure}

\begin{wrapfigure}[12]{r}[15pt]{50mm}
\includegraphics[width=50mm]{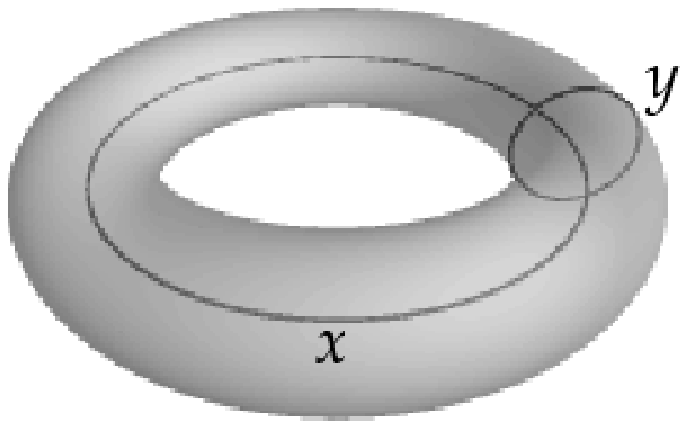}
\caption{Cell structure of the $2$--torus}
\label{torus-cells}
\end{wrapfigure}

\vbox{
\begin{corollary} \label{map_on_homology}
 The map $\alpha$ induced on integral homology by the inclusion of the boundary $\partial (_\Gamma \backslash \widehat{\Hy})$
 into the Borel--Serre compactification $_\Gamma \backslash \widehat{\Hy}$ is determined as follows.
\begin{itemize}
 \item In degree $0$, it is the augmentation map.
 \item In degree $1$, it is a surjection onto the non-cuspidal part of the homology,
       for each orbit of cusps  vanishing on exactly one of the two attached $1$--cells, 
       and injective from the set of remaining $1$--cells.
 \item In degree $2$, it is the map the kernel of which admits as generator the union of all torus $2$--cells.
\end{itemize}
\end{corollary}
}
\begin{proof}
 Consider the cellular decomposition of~$_\Gamma \backslash \Hy$ with only $3$--cell the Bianchi fundamental polyhedron.
 Consider the chain map induced by including the boundary $\partial \left(_\Gamma \backslash \widehat{\Hy}\right)$ 
 as a sub-cellular chain complex into the Borel--Serre compactification $_\Gamma \backslash \widehat{\Hy}$.
 As a sub-chain inclusion, this chain map is injective and preserves cycles.
 So we only have to divide out the boundaries in $_\Gamma \backslash \widehat{\Hy}$.
 In degree $n$, we carry this out applying Theorem~\ref{torus}~($n$).
\end{proof}

\begin{proposition}
 For any Bianchi group $\mathrm{SL_2}(\ringO)$ with only units $\{\pm 1\}$, the boundary of 
 the $2$--cell in the boundary of the Bianchi fundamental polyhedron that is contained in the real upper-half plane
 passes to the quotient by $\mathrm{SL_2}(\ringO)$ as follows.
 All of its edges in $\Hy$ appear in pairs with opposite orientation signs, so only 
 the top edge of this $2$--cell
 lying in the attached infinitely far torus contributes non-trivially,
 and it becomes a loop because its boundary vertices are identified.
\end{proposition}

\begin{proof}
 By the natural inclusion of~$\SL_2(\Z)$ into $\mathrm{SL_2}(\ringO)$,
 we obtain a natural inclusion of the upper half-plane with the action of~$\SL_2(\Z)$ as the real upper-half plane into $\Hy$,
compatible with the subgroup action of~$\SL_2(\Z)$ on $\Hy$.
 The boundary of the fundamental domain for~$\SL_2(\Z)$ in the upper half-plane
 vanishes when passing to the quotient of the action.
 The considered inclusion maps this fundamental domain to the $2$--cell in question,
 so the assertion follows.  
\end{proof}

 It follows that the
 cycle obtained from this top edge by the matrix
 $M := $\scriptsize$\begin{pmatrix} 1 & 1 \\ 0 & 1 \end{pmatrix} \quad$ \normalsize 
 identifying its endpoints,
 generates the kernel of the map induced on first degree homology by sending the $2$--torus into the compactified quotient space,
 replacing the cusp at infinity.
 Hence the matrix $M$
 is an element of the commutator subgroup of any Bianchi group.
 This can alternatively be seen from the decomposition of this matrix
as a product of finite order matrices in $\SL_2(\Z)$, which is included naturally into every Bianchi group.
Together with Theorem~\ref{torus}, we obtain the following corollary,
which has already been obtained by Serre \cite{Serre}*{page 519}.

\begin{corollary}
 The matrix \scriptsize $\begin{pmatrix} 1 & \omega \\ 0 & 1 \end{pmatrix}$ \normalsize, with $\omega$
 an element giving the ring of integers the form $\Z[\omega]$,
 represents a non-trivial element of the commutator factor group of the Bianchi group $\SL_2(\Z[\omega])$. 
\end{corollary}

\newpage
\subsection*{The long exact sequence} Let $\Gamma$ be a subgroup of finite index in a Bianchi group with only units $\{\pm 1\}$.
Let $\T_i$ be the torus attached at the cusp $i$ of~$\Gamma$,
and let $x_i$ and $y_i$ denote the cycles generating $\Homol_1(\T_i)$.
Let $c(x_i)$ be a chain of hyperbolic $2$-cells with boundary in $_\Gamma \backslash \widehat{\Hy}$ given by the cycle $x_i$.
This is well-defined because of~Theorem~\ref{torus}(1).
Let $P$ be the Bianchi fundamental polyhedron of~$\Gamma$, admitting as boundary $\cup_i \T_i$ by~Theorem~\ref{torus}(2).

\begin{corollary}
Under the above assumptions and with the above notations,
 the long exact sequence induced on integral homology by the map $\alpha$ concentrates in

\begin{tikzpicture}[descr/.style={fill=white,inner sep=1.5pt}]
        \matrix (m) [
            matrix of math nodes,
            row sep=1em,
            column sep=2.5em,
            text height=1.5ex, text depth=0.25ex
        ]
        {   &  & \Homol_3(_\Gamma \backslash \widehat{\Hy}) = 0 & \langle P \rangle \\
	    & \bigoplus_i \langle \T_i \rangle &  \bigoplus_i \langle \T_i \rangle /_{\langle \cup_i \T_i \rangle} \oplus \Homol_2^{\rm cusp} & \Homol_2^{\rm cusp}  \bigoplus_i \langle c(x_i) \rangle \\
            & \bigoplus_i \langle x_i, y_i \rangle & \bigoplus_i \langle y_i \rangle   \oplus \Homol_1^{\rm cusp} & \Homol_1^{\rm cusp} \oplus \{ { {\rm vanishing} \atop {\rm ideal} } \} \\
            & \bigoplus_i \Z  & \Z & 0, \\
        };

        \path[overlay,->, font=\scriptsize,>=latex]
        (m-1-3) edge (m-1-4) 
        (m-1-4) edge[out=355,in=175]  (m-2-2)
        (m-2-2) edge node[descr,yshift=0.9ex] {$\alpha_2$} (m-2-3)
        (m-2-3) edge (m-2-4)
        (m-2-4) edge[out=355,in=175]  (m-3-2)
        (m-3-2) edge node[descr,yshift=0.9ex] {$\alpha_1$} (m-3-3)
        (m-3-3) edge (m-3-4)
        (m-3-4) edge[out=355,in=175]  (m-4-2)
        (m-4-2) edge node[descr,yshift=0.9ex] {$\alpha_0$} node[descr,yshift=-1.2ex] {\rm augmentation} (m-4-3)
        (m-4-3) edge (m-4-4);
\end{tikzpicture}

where the maps without labels are the obvious restriction maps making the sequence exact;
 and where $\Homol_1^{\rm cusp}$ and $\Homol_2^{\rm cusp}$ are generated by cycles from the interior of~$_\Gamma\backslash \Hy$.
\end{corollary}

\begin{proof}
 We take into account that the virtual cohomological dimension of the Bianchi groups is~2,
 and apply Corollary~\ref{map_on_homology} to the long exact sequence in homology associated to the short exact sequence of chain complexes
 $\xymatrix{\partial Y \ar[r]^\alpha & Y \ar[r] & Y/_{\partial Y} }$,
 where we write $Y$ for the cellular chain complex of~$_\Gamma \backslash \widehat{\Hy}$.
\end{proof}

\begin{wrapfigure}[16]{r}[15pt]{50mm}
  \centering 
  \includegraphics[width= 35mm]{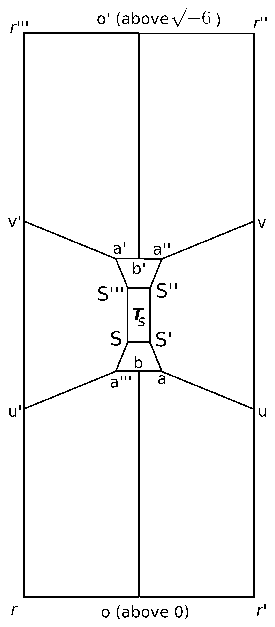}
  \caption{Bottom facets for $\SL_2({\mathcal{O}}_{\rationals(\sqrt{-6} \thinspace)})$} \label{fundamental domain 6}
\end{wrapfigure}

\begin{remarks}
\begin{enumerate}
 \item  Note that the cuspidal homology satisfies equality of ranks between $\Homol_2^{\rm cusp}$ and $\Homol_1^{\rm cusp}$,
 because the naive Euler-Poincar\'e characteristic vanishes.
\item We obtain an analogous long exact sequence in cohomology with a coefficient module $M$; and  
$\Homol^1(_\Gamma \backslash \widehat{\Hy}; \thinspace M) \cong {\rm Hom}(\bigoplus_i \langle y_i \rangle, \thinspace M) \oplus \Homol^1_{\rm cusp}$,
with $\Homol^1_{\rm cusp}$ being generated by cocycles from the interior of~$_\Gamma\backslash \Hy$.
\end{enumerate}
\end{remarks}

\subsection*{Example. The case $\Gamma = \SL_2({\mathcal{O}}_{\rationals(\sqrt{-6} \thinspace)})$.}

There are fifteen orbits of edges, labelled $(s,s')$,  $(s',s'')$, $ (r,\infty), \thinspace (\infty,\infty'), \thinspace (\infty',\infty'')$, 
$(b,a)$, $(a,s')$, $(a,u)$, $(u,v)$, $(a'',v)$, $(b,o)$, $(o',b')$, $(u,r')$, $(r'',v)$ and $(o,r')$.
The remaining edges of the Bianchi fundamental polyhedron, 
which we assemble from figures~\ref{diagram6} and \ref{fundamental domain 6},
are subject to identifications with those edges by $\Gamma = \SL_2({\mathcal{O}}_{\rationals(\sqrt{-6} \thinspace)})$. 

The boundary of the quotient of the Bianchi fundamental polyhedron is \mbox{$\T_\infty +\T_s$.}
The kernel of~$\partial_2$ is the direct sum of $\langle \T_\infty \rangle$ and $\langle \T_s \rangle$.
Therefore, \mbox{$\Homol_2(\Gamma; \rationals) \cong \rationals$.}

\begin{tabular}{l|l}
The kernel of~$\partial_1$ is & \mbox{And the image of~$\partial_2$ is} \\
\tiny
$\begin{array}{l}
      \langle (s,s')\rangle \\ 
\oplus  \langle(s',s'')\rangle \\ 
\oplus  \langle (a,u)+(u,v)+(v,a'')\rangle \\ 
\oplus  \langle (\infty,\infty')\rangle \\ 
\oplus  \langle (\infty',\infty'')\rangle \\
\oplus  \langle (r',u)+(u,v)+(v,r'') \rangle \\ 
\oplus  \langle (b,o)+(o',b')\rangle \\
\oplus  \langle (o,r')+(r',u)+(u,a)+(a,b)+(b,o)\rangle .
\end{array}$ 
\normalsize
&
\tiny 
$\begin{array}{l}\langle (s,s')\rangle \\
 \oplus  \langle (a,u)+(u,v)+(v,a'')-(s',s'')\rangle \\
 \oplus  \langle (\infty,\infty')\rangle \\
 \oplus   \langle (r',u)+(u,v)+(v,r'')-(\infty',\infty'') \rangle \\
 \oplus  \langle (o,r')+(r',u)+(u,a)+(a,b)+(b,o)\rangle  \\
 \oplus  \langle -(o,r')+(a'',v)+(v,r'')-(a,b)+(o',b') \rangle,
 \end{array}$
\normalsize
\end{tabular}

\medskip

from which we see that  $\Homol_1(\Gamma; \rationals) \cong \rationals^2$. In this case, 
this has been already 
known
by Swan~\cite{Swan},
 and later been reobtained with different methods \cite{Vogtmann}, \cite{RahmFuchs}.

\subsection*{Acknowledgements}
 I would like to thank Mathias Fuchs for helping me to obtain geometric intuition about the Borel-Serre compactification,
 and for his suggestions on the draft.
And I would like to thank the referee of a preliminary version and Philippe Elbaz-Vincent for helpful remarks.
Thanks to Nicolas Bergeron, Jean Raimbault and Shifra Reif for discussions, and to Graham Ellis for support and encouragement.
\\
This work is dedicated to the memory of Fritz Grunewald.

\end{document}